\documentclass[12pt]{article}

\usepackage{tabularx}
\usepackage{tikz}
\usepackage{relsize}
mm \textheight=8.9in

\usepackage{amssymb}
\usepackage{amsmath}
\usepackage{amsthm}
\usepackage{color}

\newcommand{\hgt}{\mathrm{ht}}

\newcommand{\br}[3]{{$#1$}$\lower4pt\hbox{$\tp\atop\raise4pt \hbox{$\scriptscriptstyle{#2}$}$} ${$#3$}}
\newcommand{\tw}[3]{{$#1$}${\,\scriptscriptstyle {#2}}\atop\raise9pt\hbox{$\scriptstyle\tp$} ${$#3$}}
\newcommand{\ttps}[2]{{#1}\raise5pt\hbox{$\lower12pt\hbox{$\scriptstyle\tp$}\atop \lower0pt\hbox{$\tilde\;$}$}\raise4.5pt\hbox{${\scriptstyle{#2}}$}}
\newcommand{\st}[1]{\mbox{${\,\scriptscriptstyle {#1}}\atop\raise5.5pt\hbox{$*$}$}}

\newcommand{\rd}[1]{\mbox{${\,\scriptscriptstyle {#1}}\atop\raise5.5pt\hbox{$\bullet$}$}}
\newcommand{\rt}[1]{\otimes_\chi}
\newcommand{\lt}[1]{\mbox{${\,\scriptscriptstyle {#1}}\atop\raise5.5pt\hbox{$\ltimes$}$}}
\newcommand{\btr}{\raise1.2pt\hbox{$\scriptstyle\blacktriangleright$}\hspace{2pt}}
\newcommand{\btl}{\raise1.2pt\hbox{$\scriptstyle\blacktriangleleft$}\hspace{2pt}}

\newcommand{\lcr}{\raise1.0pt \hbox{${\scriptstyle\rightharpoonup}$}}
\newcommand{\rcr}{\raise1.0pt \hbox{${\scriptstyle\leftharpoonup}$}}

\newcommand{\ttp}{{\lower12pt\hbox{$\tp$}\atop \hbox{$\tilde\;$}}}

\newcommand{\id}{\mathrm{id}}

\newcommand{\wt}{\mathrm{wt}}

\newcommand{\Ac}{\mathcal{A}}

\newcommand{\Jc}{\mathcal{J}}

\newcommand{\B}{{B}}

\newcommand{\Hg}{\mathfrak{H}}

\newcommand{\Ru}{\mathcal{R}}

\newcommand{\Cc}{\mathcal{C}}

\newcommand{\Vc}{\mathcal{V}}

\newcommand{\Xc}{\mathcal{X}}

\newcommand{\C}{\mathbb{C}}
\newcommand{\Z}{\mathbb{Z}}

\newcommand{\tp}{\otimes}

\newcommand{\U}{U}

\newcommand{\Fc}{\mathcal{F}}
\newcommand{\ve}{\varepsilon}
\newcommand{\gm}{\gamma}
\newcommand{\dt}{\delta}

\newcommand{\op}{\oplus}
\newcommand{\la}{\lambda}
\newcommand{\tr}{\triangleright}

\newcommand{\End}{\mathrm{End}}

\newcommand{\Span}{\mathrm{Span}}

\newcommand{\Hom}{\mathrm{Hom}}

\newcommand{\rk}{\mathrm{rk}}

\newcommand{\Rm}{\mathrm{R}}

\newcommand{\ad}{\mathrm{ad}}

\newcommand{\La}{\Lambda}
\newcommand{\Ga}{\Gamma}
\newcommand{\g}{\mathfrak{g}}
\renewcommand{\a}{\mathfrak{a}}
\renewcommand{\b}{\mathfrak{b}}
\renewcommand{\k}{\mathfrak{k}}
\newcommand{\h}{\mathfrak{h}}

\newcommand{\eps}{\epsilon}

\newcommand{\nn}{\nonumber}

\renewcommand{\l}{\mathfrak{l}}
\renewcommand{\c}{\mathfrak{c}}

\newcommand{\al}{\alpha}
\newcommand{\bt}{\beta}

\newcommand{\prt}{\partial}
\newcommand{\nbl}{\nabla}

\newcommand{\be}{\begin{eqnarray}}
\newcommand{\ee}{\end{eqnarray}}

\newtheorem{thm}{Theorem}[section]
\newtheorem{propn}[thm]{Proposition}
\newtheorem{lemma}[thm]{Lemma}

\newtheorem{remark}[thm]{Remark}
\newtheorem{definition}[thm]{Definition}

\newcount\prg

\newcommand{\parag}{\advance\prg by1 {\noindent\bf\thesection.\the\prg\hspace{6pt}}}

\begin{document}
\title{Mickelsson algebras via Hasse diagrams}

\author{
Andrey Mudrov${}^{\dag,\ddag}$ and Vladimir Stukopin${}^\ddag$
\vspace{10pt}\\
\small ${}^{\dag}$ University of Leicester, \\
\small University Road,
LE1 7RH Leicester, UK,
\vspace{10pt}\\
\small
${}^{\ddag}$ Moscow Institute of Physics and Technology,\\
\small
9 Institutskiy per., Dolgoprudny, Moscow Region,
141701, Russia,
\vspace{10pt}\\
\small
 e-mail:  am405@le.ac.uk, stukopin.va@mipt.ru
\\
Corresponding author: Andrey Mudrov}
\date{ }
\maketitle

\begin{abstract}
Let $\mathcal{A}$ be an associative algebra containing either classical or quantum universal enveloping algebra of
 a semi-simple complex Lie algebra $\mathfrak{g}$. We present a construction of the Mickelsson algebra $Z(\mathcal{A},\mathfrak{g})$
 relative to the left ideal in $\mathcal{A}$ generated by positive root vectors. Our method employs a calculus
 on Hasse diagrams associated with classical or quantum $\mathfrak{g}$-modules.
 We give an explicit expression for  a  PBW basis in $Z(\mathcal{A},\mathfrak{g})$ in the case when $\mathcal{A}=U(\mathfrak{a})$ of a finite-dimensional Lie algebra  $\mathfrak{a}\supset \mathfrak{g}$.
 For  $\mathcal{A}=U_q(\mathfrak{a})$ and $\mathfrak{g}$ the commutant of a Levi subalgebra in $\mathfrak{a}$, we construct a PBW basis  in terms of
 quantum Lax operators, upon extension of the ground ring of scalars to $\mathbb{C}[[\hbar]]$.
\end{abstract}
\begin{center}
\end{center}
\\
{\small \underline{AMS classification codes}: 17B10, 17B37.}

\newpage

\section{Introduction}
\label{SecIntro}
Step algebras were invented by Mickelsson \cite{Mick} and intensively developed by Zhelobenko \cite{Zh2}.
They find interesting applications
in representation theory like Gelfand-Tsetlin basis \cite{Mol}, Harish-Chandra modules \cite{Mick,Zh1}, Yangians \cite{KN1,KN2,KN3} {\sl etc}.

Mickelsson algebras deliver an example  of reduction algebras $R(\Ac,\Jc)$ associated with a pair
of an associative algebra $\Ac$ and a left ideal $\Jc\subset \Ac$. The  algebra $R(\Ac,\Jc)$ is
a quotient $N(\Jc)/\Jc$ of the normalizer $N(\Jc)$  of $\Jc$, i.e. the maximal subalgebra in $\Ac$ where
$\Jc$ is a two-sided ideal. Multiplication  by $z\in N(\Jc)$ on the right preserves $\Jc$ and  induces an endomorphism of the $\Ac$-module $\Ac/\Jc$.
Furthermore, $R(\Ac,\Jc)$ is naturally isomorphic to the kernel of $\Jc$ in $\Ac/\Jc$.
In general,  the kernel of $\Jc$ in every $\Ac$-module $V$ parameterizes $\Hom_\Ac(\Ac/\Jc,V)$ and is preserved by $R(\Ac,\Jc)$.

When $\Ac$ contains a universal enveloping algebra  of a semi-simple Lie algebra $\g$ and $\Jc$ is generated by
its positive nilpotent subalgebra relative to a polarization $\g=\g_-\op\h\op \g_+$, $R(\Ac,\Ac\g_+)$ is called  Mickelsson algebra and is denoted by $Z(\Ac,\g)$.
It is responsible
for restriction of $\Ac$-modules to $\g$-modules and  can be viewed as the symmetry algebra of $\g$-singular
vectors, that parameterize $\g$-submodules in a wide class of  $\Ac$-modules.
One can similarly define a Mickelsson algebra $Z_q(\Ac,\g)$ upon replacement of $U(\g)$ with the quantum group $U_q(\g)$.

The main tool for construction  and study of Mickelsson algebras  is  a theory of extremal projector \cite{AST}, after the works of Zhelobenko, \cite{Zh2}.
That  is an element of a certain extension of the (quantized) universal enveloping algebra of $\g$ presentable as  a finite  product of infinite series in root vectors.
Its action on a general element of $\Ac$ is complicated and   explicitly computable in relatively simple cases.
 Our approach to Mickelsson algebras  employs transformation properties  of Lax operators and combinatorics of  Hasse diagrams \cite{Birk} associated with classical or quantum $\g$-modules.
This  can be viewed as an indirect although constructive  way of computing the action of the extremal projector.

The method of Hasse diagrams was initiated in \cite{M1} as a  development of a Nagel-Moshinsky construction \cite{NM} for lowering
and rasing operators for $\g\l(n)$, which technique had been  a precursor to step algebras.
The formalism of Hasse diagrams has lead to an explicit expression for the inverse Shapovalov form in terms of quantum Lax operators.
It turns out that  a modified algorithm yields an explicit formulation of Mickelsson algebras.
It was applied in \cite{AM} to a special case of isotypic, co-rank 1, reductive pairs of non-exceptional quantum groups.
Such pairs appear in the theory of Gelfand-Tsetlin basis, \cite{Mol}.
In this paper, we develop this formalism in a systematic way.

 \section{Quantum group basics}
\label{SecQuantGroups}

Let  $\g$ be a  semi-simple complex Lie algebra and  $\h\subset \g$ its Cartan subalgebra. Fix
a triangular decomposition  $\g=\g_-\op \h\op \g_+$  with  maximal nilpotent Lie subalgebras
$\g_\pm$.
Denote by  $\Rm \subset \h^*$ the root system of $\g$, and by $\Rm^+$ the subset of positive roots with basis $\Pi$
of simple roots. This basis  generates a root lattice $\Ga\subset \h^*$ with the positive semigroup $\Ga_+=\Z_+\Pi\subset \Ga$.
The height $\hgt(\mu)$ of $\mu\in \Ga_+$ is the number of simple roots entering $\mu$.
The Weyl vector $\frac{1}{2}\sum_{\al\in \Rm^+}\al$ is denoted by $\rho$.

Choose the  Killing form $(\>.\>,\>.\>)$ on $\g$, restrict it to $\h$, and transfer to $\h^*$ by duality.
For every $\la\in \h^*$ there is   a unique element $h_\la \in \h$ such that $\mu(h_\la)=(\mu,\la)$, for all $\mu\in \h^*$.

We assume that a deformation parameter  $q\in \C$ takes values away from  roots of unity.
The  quantum group $U_q(\g)$   is a complex Hopf algebra with the set of generators $e_\al$, $f_\al$, and $q^{\pm h_\al}$ labeled by $\al\in \Pi$ and satisfying relations \cite{D1,ChP}
$$
q^{h_\al}e_\bt=q^{ (\al,\bt)}e_\bt q^{ h_\al},
\quad
[e_\al,f_\bt]=\dt_{\al,\bt}[h_\al]_q,
\quad
q^{ h_\al}f_\bt=q^{-(\al,\bt)}f_\bt q^{ h_\al},\quad \forall \al, \bt \in \Pi.
$$
The symbol $[z]_q$, where  $z\in \h+\C$, stands for $\frac{q^{z}-q^{-z }}{q-q^{-1}}$ throughout the paper.
The elements $q^{\pm h_\al}$ are inverse to each other, while  $\{e_\al\}_{\al\in \Pi}$ and $\{f_\al\}_{\al\in \Pi}$ also
satisfy quantized Serre relations. Their exact form can be found in  \cite{ChP}; it is not important for this presentation.

A Hopf algebra structure on $U_q(\g)$ is fixed by the comultiplication
$$
\Delta(f_\al)= f_\al\tp 1+q^{-h_\al}\tp f_\al,\quad \Delta(q^{\pm h_\al})=q^{\pm h_\al}\tp q^{\pm h_\al},\quad\Delta(e_\al)= e_\al\tp q^{h_\al}+1\tp e_\al,
$$
set up on the generators and extended as an algebra homomorphism $U_q(\g)\to U_q(\g)\tp U_q(\g)$.
The antipode $\gm$ is an algebra and coalgebra anti-automorphism of $U_q(\g)$ that acts on the generators by the assignment
$$
\gm( f_\al)=- q^{h_\al}f_\al, \quad \gm( q^{\pm h_\al})=q^{\mp h_\al}, \quad \gm( e_\al)=- e_\al q^{-h_\al}.
$$
The counit homomorphism $\eps\colon U_q(\g)\to \C$ returns on the generators the values
$$
\eps(e_\al)=0, \quad \eps(f_\al)=0, \quad \eps(q^{h_\al})=1.
$$

There are subalgebras  $U_q(\h)$,  $U_q(\g_+)$, and $U_q(\g_-)$    in $U_q(\g)$
 generated by $\{q^{\pm h_\al}\}_{\al\in \Pi}$, $\{e_\al\}_{\al\in \Pi}$, and $\{f_\al\}_{\al\in \Pi}$, respectively.
The quantum Borel subgroups are defined as $U_q(\b_\pm)=U_q(\g_\pm)U_q(\h)$; they are Hopf subalgebras in $U_q(\g)$.
The algebra $U_q(\g)$ features a triangular factorization $U_q(\g_-)\tp U_q(\h)\tp U_q(\g_+)\stackrel{\simeq}{\longrightarrow} U_q(\g)$, similarly to $U(\g)$.
This is an isomorphism of vector spaces implemented by taking product of the tensor factors.

Given a $U_q(\h)$-module $V$, a non-zero vector $v$ is said to be of weight $\mu$ if $q^{h_\al}v=q^{(\mu,\al)} v$ for all $\al\in \Pi$.
The linear span of such vectors is denoted by $V[\mu]$. The set of weights of a module $V$ is denoted by $\La(V)$.
A $U_q(\g)$-module $V$ is said to be of highest weight $\la$ if it is generated by a weight vector $v\in V[\la]$ that
is killed by all $e_\al$. Such a vector $v$ is called highest; it is defined up to a non-zero scalar multiplier.

We choose a quasitriangular structure (universal R-matrix) $\Ru$ of $U_q(\g)$ so that every finite dimensional representation of its left (resp. right) tensor factor
produces a matrix with entries in $U_q(\b_-)$ (resp. $U_q(\b_+)$).

\section{Mickelsson algebras}
\label{SecMickAlg}
In this section we recall the definition and  basic facts about Mickelsson algebras.

Let $\g$ be a semi-simple Lie algebra  with a polarization $\g=\g_-\op \h \op \g_+$ and $U$ denote either its classical or quantized universal enveloping algebra
(quantum group).
Let $U^-U^0U^+$ be the triangular factorization relative to the polarization. The subalgebras $U^\pm$ are endowed with natural grading via setting
$\deg=1$ on the simple root generators. We denote by $B^\pm=U^\pm U^0= U^0 U^\pm$ the Borel subalgebras in $U$.

Suppose that $U$ is contained in an associative algebra $\Ac$. We assume that $\Ac$ is diagonalizable under the adjoint action of $U^0$.
Define a left ideal $\Jc=\Jc_+=\Ac\g_+\subset \Ac$ as the one generated by $e_\al$ with $\al \in \Pi$.
A right ideal $\Jc_-=\g_- \Ac$ is defined similarly.

Denote by $N(\Jc)$ the normalizer of $\Jc$ in $\Ac$. It is
the set of elements  $a\in \Ac$ such that $\Jc a \subset \Jc$. It is clear that $N(\Jc)$ is an algebra with unit and $\Jc$ is a two-sided ideal in it.

\begin{definition}
  The quotient $N(\Jc)/\Jc$ is called Mickelsson algebra of the pair $\Ac\supset U$ and denoted by $Z(\Ac,\g)$.
\end{definition}
In the special case when $\Ac$ is either classical or quantized universal enveloping algebra
of a simple Lie algebra $\a\supset\g$,  we use the notation $Z(\a,\g)$ or $Z_q(\a,\g)$ when we wish to emphasize the quantum group setting.

Consider an extension of the subalgebra $U^0$ to the ring $\hat U^0$ of fractions over $[h_\mu-c]_q$ (resp. $h_\mu-c$ in the classical case)
where $\mu$ is an integral weight  and the range of $c\in \C$  depends on set of weights of $\Ac$.
We assume that $\Ac$ is diagonalized  under the adjoint  $\hat U^0$-action and all weights are integral. Then we can
similarly  extend $\Ac $ to an algebra $\hat \Ac=\Ac\tp_{U^0}\hat U^0$, and  $\Jc_\pm$ to ideals $\hat\Jc_\pm\subset \hat \Ac$.

We suppose that there exist further extensions  $\hat {U}\subset \breve{\U}$, $\hat {\Ac}\subset \breve{\Ac}$,
$\hat \Jc_\pm\subset \breve{ \Jc}_\pm$
such that  $\Jc_\pm= \breve{ \Jc}_\pm\cap \Ac$, and there is an element (extremal projector) $\wp\in \breve  { \U}$ of zero weight satisfying
\be
\wp \Jc_-=0=\Jc_+\wp, \quad \wp^2=\wp, \quad \wp-1\in \breve { \Jc}_-\cap \breve  {\Jc}_+.
\label{extr_proj}
\ee
Such extensions do exist for $\Ac$  the (quantized) universal enveloping algebra of a Lie algebra containing $\g$.
The exact formulas for $\wp$ can be found in $\cite{K,KT}$.

The inclusion in the right formula of (\ref{extr_proj}) implies the following isomorphisms of  $\hat U^0$-bimodules:
$$
\hat \Ac/\hat \Jc_+ \simeq \hat \Ac\wp, \quad  \hat \Ac/\hat \Jc_-\simeq  \wp \hat \Ac, \quad
\hat \Ac/(\hat \Jc_-+\hat\Jc_+) \simeq  \wp \hat \Ac\wp.
$$
We denote by $\Vc$ the quotient $\hat \Ac/\hat \Jc_+$ and regard it as an $\Ac$-module. By $\Vc^+$ we denote the kernel of the ideal $\Jc$ in $ \Vc$, that is, the subspace of $U^+$-invariants.
It is a  natural $\hat U^0$-bimodule.

Define a map $\varphi\colon  \hat \Ac\to \breve{\Ac}$ by the assignment
$$
x\mapsto \wp x\wp.
$$
\begin{thm}[\cite{K}]
\label{thmMick}
The map $\varphi$ induces  an isomorphism of algebras
$$
\hat Z(\Ac,\g)\to \wp \hat \Ac \wp.
$$
\end{thm}
From now on we suppose that $\Ac$ is a free $U^- -  B^+$-bimodule relative to the regular left-right actions.
Suppose that  $\Xc\subset \Ac$ is a vector space   such that $\hat \Ac=U^-\Xc \hat B^+$.
Then $\Xc \hat U^0$ is isomorphically mapped onto the double coset algebra $\hat \Ac/(\hat J_-+ \hat J_+)$, \cite{K}.

We say that $\Xc\subset \Ac$ delivers a Poincare-Birkhoff-Witt (PBW) basis over $U$ if
there is a set of weight elements $\{a_i\}_{i=1}^k\subset \Xc$ such that ordered
monomials $a_1^{m_1}\ldots a_k^{m_k}\subset \Xc$ freely generate   $\Ac$ as a regular $U^- - \hat \B^+$-bimodule.
The set $\{a_i\}_{i=1}^k$ is called a PBW system and
its elements are called PBW generators. The images of PBW monomials in the double coset algebra $\hat \Ac/(\hat \Jc_-+\hat\Jc_+) $
deliver a basis over $\hat U^0$. It is known that the $\varphi$-image of a PBW-system
in $\wp \hat \Ac\wp\simeq $ generates a PBW-basis in $\hat Z(\Ac,\Jc)$ over $\hat U^0$, cf. \cite{K}.

The extremal projector is an ordered  product of factors, each of the form
$$
\wp_\al=\sum_{k=0}^{\infty }f_\al^k e_\al^k \wp_{\al,k}, \quad \wp_{\al,k}\in \hat U^0, \quad \al \in \Rm^+.
$$
Here $f_\al$ and $e_\al$ are root vectors associated with a normal ordering on $\Rm^+$, \cite{KT}.
For $a\in \Ac/\Jc$, pushing positive root vectors in $\wp a$  to the right  until they  are annihilated by $\Jc$
produces a complicated expression which  in general is hard to evaluate directly.
In this  paper we compute this action using a bypassing technique, under an  assumption that the PBW generators span
a finite dimensional $\B^+$-module.

 \section{Quantum Lax operators}
\label{SecQuantLax}
It is possible \cite{ChP} to choose a quasitriangular structure  for the quantum group $U$ in the form
$$
\Ru=q^{\sum_{i}h_i\tp h_i}\check{\Ru}, \quad \mbox{with}\quad \check{\Ru}\in U^+\tp \U^-,
$$
where $\{h_i\}_{i=1}^{\rk \g}$ is an orthonormal basis in $\h$  (here $\rk\> \g=\dim \h$ is the rank of $\g$).
The matrix $\check{\Ru}$  intertwines two comultiplications on $U$:
$$
\check{\Ru}\Delta(x)=\tilde \Delta(x)\check \Ru, \quad \forall x\in U,\quad \mbox{where}
$$
$$
\tilde \Delta(f_\al)=f_\al\tp 1 +q^{h_\al}\tp f_\al,\quad \tilde \Delta (q^{\pm h_\al})=q^{\pm h_\al}\tp q^{\pm h_\al}, \quad \tilde \Delta(e_\al)=e_\al\tp q^{-h_\al}+1\tp e_\al.
$$
Define
$$
\Fc=\frac{1}{q-q^{-1}}(\check \Ru -1\tp 1)\in U^+\tp \U^-.$$
 The intertwining relations for positive generators
translate to
\be
\label{int-rel}
(e_{\al}\tp q^{-h_{\al}} + 1\tp e_{\al}) \Fc-\Fc (e_{\al}\tp q^{h_{\al}} + 1\tp e_{\al})=e_{\al}\tp [h_{\al}]_q  , \quad \forall\al\in \Pi.
\ee
This is a key relation that underlies a calculus on Hasse diagrams. In the classical limit it becomes
$$
 [e_\al \tp 1+1\tp e_\al,\Cc] = e_\al\tp h_\al, \quad \al\in \Pi,
$$
where $\Fc$ degenerates to $\Cc=\sum_{\al\in \Rm^+}e_\al\tp f_\al$.

\begin{definition}
  We call a $U^+$-module graded if the action extends to an $B^+$-action diagonalizable over $U^0$ with finite dimensional weight spaces.
\end{definition}
\noindent
Note that such an extension is not unique because all weights can be shifted by some $\la\in \h^*$
(cf. e.g. Verma modules of different highest weights).

We assume the standard partial order on weights  of a graded $U^+$-module: $\nu_i>\nu_j$ if and only if  $\nu_i-\nu_j\in \Z_+\Pi$.

Let $X$ be a graded  $U^+$-module with representation homomorphism $\pi\colon U^+\to \End(X)$ and denote by $\pi^\al_{ij}$ the matrix entries $\pi(e_\al)_{ij}$ in a fixed weight basis $\{x_i\}_{i\in I_X}\subset X$.
We assume  that weights of $X$ are bounded from below.
Introduce a matrix
$$
F=\sum_{i,j}e_{ij}\tp \phi_{ij}=(\pi\tp \id)(\Fc)\in \End(X)\tp U^-,
$$
where $e_{ij}$ are standard matrix units, $e_{ij}x_k=\dt_{jk}x_i$.
The intertwining relations (\ref{int-rel}) reduce to equalities on the matrix entries:
\be
e_{\al}\phi_{ij}- \phi_{ij} e_{\al}=  \sum_{k\in I_X}\phi_{ik}q^{h_{\al}} \pi^\al_{kj}  -\sum_{k\in I_X} \pi^\al_{ik} q^{-h_{\al}}\phi_{kj}+\pi^\al_{ij}[h_\al]_q
\quad
\al\in \Pi.
\label{intertwiner_F}
\ee
 Note that the $F$ is strictly  lower triangular: $\phi_{ij}\not =0$ only if $\nu_i>\nu_j$; the weight
 $\wt(\phi_{ij})$  of the element $\phi_{ij}$ is $\nu_j- \nu_i<0$.

Suppose that $X$ is such that $X'\subset \Vc$ is a $B^+$-module isomorphic to the left dual to $X$ with respect to the comultiplication
$\tilde \Delta$. If  $\{\psi_i\}_{i\in I_X}\subset X'$ is the dual basis to a weight basis $\{x_i\}_{i\in I_X}\subset X$,
then the tensor $\psi_X=\sum_{i\in I_X}x_i\tp \psi_i$ is $B^+$-invariant. We call it canonical invariant relative to $X$.

We consider $\psi_X$ as a vector with components $\psi_i\in \Vc$. By definition, they transform under the action of $U^+$ as
\be
  e_\al \psi_i   =  -\sum_{k\in I_{X}} q^{-h_\al}\pi^\al_{ik}\psi_{k}, \quad \forall i\in I_{X}.
\label{can_invariant}
\ee
Note that  the weight of $\psi_i$ is $-\nu_i=-\wt(x_i)$.

\section{Hasse diagrams associated with representations}
\label{SecHasseDiag}

Recall that Hasse diagrams are oriented graphs associated with partially ordered sets.
Nodes are elements of the set, and arrows $i\leftarrow j$ are ordered pairs of nodes $(i,j)$, $i\succ j$ such that there is no other node between them.
We will associate a  Hasse diagram   with any graded $U^+$-module.

Suppose that  $X$ is such a module that $X'\subset \Vc$, as discussed in the previous section. Let $\pi\colon U^+\to \End(X)$ denote the representation homomorphism. Fix a weight
 basis $\{x_i\}_{i\in I_X}\subset X$ and let $\nu_i$ denote the weights of $x_i$. The matrix entry  $\pi^\al_{ij}$ is  distinct from zero  only if $\nu_i-\nu_j=\al$.

The  nodes  of Hasse diagram $\Hg(X)$ are defined to be basis elements $\{x_i\}_{i\in I_X}$. We will also identify them with elements of the index set $I_X$.

 Arrows  are labeled by simple roots $\al\in \Pi$:
we write $i\stackrel{\al}{\longleftarrow} j$ if $\pi^\al_{ij}\not =0$.
A sequence of adjacent arrows
$$
m_1
\stackrel{\al_1}{\longleftarrow} m_2\stackrel{\al_2}{\longleftarrow} \ldots \stackrel{\al_k}{\longleftarrow} m_{k+1}
$$
is called a path (of length $k$) from $m_{k+1}$ to $m_1$.
The partial order on $I_X$ is defined as follows: a node $i$ is superior to a  node $j$, $i\succ j$, if there is a path from $j$ to $i$.

There is a weaker partial order on $I_X$ induced by the standard partial order $>$ on their weights.
We have earlier noted that the matrix $F$ is lower triangular with respect to $>$. We will need a refinement of this property
with respect to   $\succ$.
\begin{propn}
\label{strong-triangular}
  For all $i,j\in I_X$, $\phi_{ij}\not=0\Rightarrow i\succ j$.
\end{propn}
\begin{proof}
Monomials in simple positive root vectors span $U^+$. The matrix entry
$\pi(e_{\al_1}\ldots e_{\al_k})_{ij}$ of such a monomial of weight $\nu_i-\nu_j=\sum_{l=1}^{k}\al_l$
equals $\sum_{i_0,\ldots, i_{k}} \pi^{\al_1}_{i_0i_1}\pi^{\al_2}_{i_1i_2}\ldots \pi^{\al_k}_{i_{k-1}i_k}$
with $i=i_0$ and $i_k=j$.
If  there is a non-zero summand in this sum, then all its factors are non-zero, and
the sequence $i_0\succ \ldots \succ i_k$ is a path to $i$ from $j$.

We conclude that $i\not \succ j$ forces   $\pi(u)_{ij}=0$ for  any  $u\in U^+$.
Let $\{y_l\}_{l\in I_X}$ be the dual basis  in the vector space of linear functions on  $X$.
We can write $\bigl(y_i,\pi(u) x_j\bigr)=\pi(u)_{ij}=0$
for all  $u\in U^+$.
Then $\phi_{ij}=\bigl(y_i,\pi(\check{\Ru}_1)x_j\bigr)\tp \check{\Ru}_2=0$ as required.
\end{proof}

A pair of nodes $(l,r)$ is called simple of there is an arrow from $r$ to $l$; then $\nu_l-\nu_r=\al \in \Pi$.
We denote the set of such pairs by $P(\al)$, it is in bijection with the subset of arrows in $\Hg(X)$ labelled by $\al$.

A route $\vec m =(m_1,\ldots m_k)$ from $j$ to $i$ is an arbitrary ordered sequence
$$
i=m_1\succ \ldots \succ m_k=j.
$$
We denote $\max(\vec m)=i $ and $\min(\vec m)=j$ and say
that $\vec m$ is a route $i\dashleftarrow j$ and  present it as $i\stackrel{\>\>\vec m}{\dashleftarrow} j$.
We will also suppress one of the nodes when we wish to emphasize the start or end node: for instance, $i\stackrel{\>\>\vec m}{\dashleftarrow} $ means that $\vec m$ is route $i\dashleftarrow$
(terminating at $i$).

 The integer $|\vec m|=k-1$ is called length of $\vec m$.
Thus a path to $i$ from  $j$ is a route $i\dashleftarrow j$ of maximal length equal to $\hgt(\nu_i-\nu_j)$.

We  orient  routes so  that nodes ascend from right to left. The reason will be clear in the next section where we
 assign a product of non-commutative matrix entries of quantum Lax operators to every route.

We will introduce  two partial operations on routes.
If  $\min(\vec m)\succ \max(\vec n)$, we will write $\vec m\succ \vec n$. There is a route that includes all nodes from $\vec m$ and $\vec n$.
We denote the resulting route  $(\vec m,\vec n)$ and call it union of $\vec m$ and $\vec n$. We will drop the vector superscript for routes consisting of one node.
For instance, $(\vec m, \vec n)=(i,j,k)$ if $\vec m=(i)\succ (j,k)=\vec n$.

The second operation is also the "union" of routes but it is defined if the start node of the left one coincides with the end
node of the right one. That is, given two routes $i\stackrel{\>\>\vec m}{\dashleftarrow} k\stackrel{\>\>\vec n}{\dashleftarrow} j$ we get
a route $i\stackrel{\>\>\vec m\cdot \vec n}{\dashleftarrow} j$ which is the concatenation of   $\vec m$ and $\vec n$.

It is clear that removing an arbitrary subset of nodes from a route is a route again (possibly empty).

\section{Localization of $U^+$-action}
In order to study the  $U^+$-action on $\Vc$, we introduce a pair of  auxiliary right $\hat U^0$-modules $\Phi_X$, $\Psi_X$, and a system of local operators
$\nbl_{l,r}\in \Hom_{\hat U^0}(\Psi_X, \Phi_X\tp_{\hat U^0} \Psi_X)$ for
all simple pairs $(l,r)$  of nodes (equivalently, all arrows) in the Hasse diagram $\Hg(X)$.
They will play a role of matrix units, and the operator $e_\al$ will be lifted to $\Hom_{\hat U^0}( \Psi_X,\Phi_X\tp_{\hat U^0} \Psi_X)$ as an expansion over $\nbl_{l,r}$ with ${(l,r)\in P(\al)}$.

For each weight $\mu\in \h^*$ put
\be
\eta_{\mu}=h_\mu+(\mu, \rho)-\frac{1}{2}(\mu,\mu) \in \h\op \C.
\ee
We will also use shortcuts $\eta_{ij}=\eta_{\nu_i-\nu_j}$ and $\eta_i=\eta_{\nu_i}$, for $i,j\in I_X$.

 \begin{lemma}
\label{thetas}
Let  $(l,r)\in P(\al)$ for some $\al\in \Pi$. Then, for all $j\in I_X$,
\be
 \eta_{lj}-\eta_{rj}&=&h_\al+(\al,\nu_j-\nu_r),
\label{thetas1}
\\
\eta_{l}-\eta_{r}&=&\eta_{lj}-\eta_{rj}-(\al,\nu_j)=h_\al-(\al,\nu_r).
\label{thetas2}
\ee
\end{lemma}
\begin{proof}
We have $(\al, \rho)=\frac{1}{2}|\!|\al|\!|^2$ for all $\al \in \Pi$ and $\nu_l-\nu_r=\al$.  Then
\be
\eta_{lj}-\eta_{rj}&=&h_\al+\frac{1}{2}|\!|\al|\!|^2-\frac{1}{2}|\!|\nu_j-\nu_r-\al|\!|^2+\frac{1}{2}|\!|\nu_j-\nu_r|\!|^2=h_\al+(\al,\nu_j-\nu_r),
\nn
\ee
which proves (\ref{thetas1}). Using this equality, we find
$$
\eta_{lj}-\eta_{rj}-(\al,\nu_j)=h_\al+(\al,\nu_j-\nu_r)-(\al,\nu_j)=h_\al-(\al,\nu_r)=\eta_{l}-\eta_{r}.
$$
This proves (\ref{thetas2}).
\end{proof}

Define $\Phi_X$ to be  a free right  $\hat U^0$-module generated by all routes in $\Hg(X)$.
To every route  $i\dashleftarrow j$   we assign a weight $\nu_j-\nu_i$.
This makes $\Phi_X$ a left $U^0$-module and therefore a $\hat U^0$-bimodule.
The left action is defined  as $h\> \vec m =\vec m \>\tau_{\wt(\vec m)}(h)$ for $h\in U^0$,
where $\tau_\mu$ is an automorphism of the algebra $U^0$ induced by the translation $\h^*\to \mu+\h^*$ along  $\mu\in \h^*$.
It assigns $q^{h_\al}\mapsto q^{h_\al}q^{(\al,\mu)}$ for all $\al \in \Pi$.


For each route $\vec m$ from $\Hg(X)$ we define an element
$$
 \phi_{\vec m}= \phi_{m_1,m_2} \ldots  \phi_{m_{k-1},m_k}\in U^-
$$
end extend this assignment to a $U^0$-bimodule homomorphism $p_\Phi\colon \Phi_X\to B^-$.
For a route $(i)$ of zero length we set $\phi_i=1$.
The map $p_\Phi$ is multiplicative with respect to concatenation of routes:
$$
p_\Phi( \vec m\cdot \vec n)=p_\Phi( \vec  m)p_\Phi( \vec n), \quad \vec  m,\vec  n\in \Phi_X.
$$
Now suppose that $X$ is a graded $U^+$-module and $\psi_X=\sum_{i\in I_X}x_i\tp \psi_i \in X\tp  \Vc$ the corresponding canonical element.
We define a free right $\hat U^0$-module $\Psi_X$ similarly to $\Phi_X$ but now we assign
 a weight $-\nu_i$ to every route $\vec m\in \Psi_X$, with $i=\max(\vec m)$, the leftmost node in $\vec m$.
This makes $\Psi_X$ a   $\hat U^0$-bimodule similarly to $\Phi_X$.

Define a $\hat U^0$-bimodule map $p_\Psi\colon \Psi_X\to \Vc$ by
$$
\vec m=(m_1,\ldots, m_k)\mapsto \phi_{m_1,m_2}\ldots \phi_{m_{k-1},m_k}\psi_{m_k}=\phi_{\vec m}\psi_{m_k}=\psi_{\vec m}.
$$
 Note that
$$
p_\Psi( \vec m\cdot \vec n)=p_\Phi( \vec m)p_\Psi( \vec n), \quad \vec m\in \Phi_X,\quad \vec n\in \Psi_X,\quad
$$
i.e. the map $p_\Psi$ is multiplicative with respect to concatenation of routes.

We define an operator $\prt_{l,r}\colon \Phi_X\to \Phi_X\tp_{\hat U^0} \Phi_X$ for $(l,r)\in P(\al)$ as follows.
We set it zero on routes of zero length.
On routes of length $1$ we put
\be
\begin{array}{rrccc}
\prt_{l,r} (l,r)&=&(l)\tp (r)[h_\al]_q  ,\\
\prt_{l,r} (l,j)&=&- (l)\tp q^{-h_\al}(r,j),
&r\succ j,
\\
\prt_{l,r}(i,r) &=&(i,l)q^{h_\al}\tp (r),
&i\succ l,\\
\prt_{l,r}(i,j) &=&0&\mbox{otherwise}.
\end{array}
\label{prt on phi}
\ee
We extend it to all routes as  a homomorphism of right $\hat U^0$-modules and a derivation with respect to concatenation:
$$
\prt_{l,r}(\vec m\cdot \vec n)=(\prt_{l,r} \vec m)\cdot \bigl((l)\tp  \vec n\bigr) + \bigl(\vec m\tp (r)\bigr)\cdot (\prt_{l,r} \vec n).
$$
Here we assume that concatenation commutes with the $\hat U^0$-actions in the following sense:
$(\vec \ell \>h)\cdot \vec \rho=\vec \ell \cdot (h\>\vec \rho)$ and  $\vec \ell \cdot (\vec \rho \>h)=(\vec \ell \cdot \vec \rho) h $ for all $h\in \hat U^0$.
At most one concatenation factor survives the action of $\prt_{l,r}$; then $\cdot$ in the right-hand side makes sense.

Now we define operators  $\nbl_{l,r}\colon\Psi_X\to \Phi_X\tp_{\hat U^0} \Psi_X$.
On routes of zero length,  we put
\be
\nbl_{l,r}(j)=
-\dt_{lj} (l)\tp  q^{-h_\al}(r)=-\dt_{lj}(l) \tp  (r)q^{-(\eta_{l}-\eta_{r})}.
\label{prt on psi}
\ee
A general route $\vec m\in \Psi_X$ can be presented as a concatenation $\vec m\cdot (j)$ with  $j=\min(\vec m)$ and $\vec m$  viewed as an element of $\Phi_X$.
Set
$$
\nbl_{l,r}(\vec m)=(\prt_{l,r}\vec m)\cdot \bigl((l)\tp  (j)\bigr) +\bigl(\vec m\tp (r)\bigr)\cdot \bigl(\nbl_{l,r}(j)\bigr)
$$
and extend it to $\Psi_X$ as a right $\hat U^0$-module homomorphism.
\begin{lemma}
\label{diff1}
The operator $\nbl_{l,r}$ is  a derivation with respect to concatenation,
$$
\nbl_{l,r}( \vec m\cdot \vec n)=(\prt_{l,r} \vec m)\cdot \bigl((l)\tp \vec n\bigr) + \bigl(\vec m\tp (r)\bigr)\cdot (\nbl_{l,r}\vec n).
$$
\end{lemma}
\begin{proof}
  Clear.
\end{proof}
We describe the action of $\nbl_{l,r}$ in more detail. Let $\vec m$ be a route with $i=\max(\vec m)$. We shall consider the following three cases depending
on a position of $(l,r)$ relative to  $\vec m$.

1.  $\nbl_{l,r}\vec m=0$ if $\vec m\cap (l,r)=\varnothing$.

2. $\nbl_{l,r}\vec m=0$ if a route $\vec m$ does intersect with $(l,r)$ but $\vec m \cup (l,r)$ is not a route $i\dashleftarrow $.

3. Suppose that  a route $i\stackrel{\>\>\vec m}\dashleftarrow $  with $|\vec m|>0$ intersects with $(l,r)$ and $\vec m\cup (l,r)$ is also a route
$i\dashleftarrow $.
Define $\vec \ell$ as consisting of nodes from $\vec m$ on the left of $l$, and $\vec \rho$ of nodes on the right of $r$.
In particular, they may be empty. Then, for $(l,r)\in P(\al)$, we have
\be
\begin{array}{rrccc}
\nbl_{l,r}\bigl((\vec\ell,l)\cdot (l,r)\cdot(r,\vec \rho)\bigr)&=&(\vec\ell,l) \tp [h_\al]_q (r,\vec \rho)&=&
(\vec\ell,l)\tp (r,\vec \rho)[\eta_{l}-\eta_{r}]_q ,\\
\nbl_{l,r}\bigl((\vec\ell,l)\cdot (l,\vec \rho)\bigr)&=&-(\vec\ell,l)\tp  q^{-h_\al}(r,\vec \rho)
&=&-(\vec\ell,l)\tp (r,\vec \rho)q^{-(\eta_{l}-\eta_{r})},
\\
\nbl_{l,r}\bigl((\vec\ell,r)\cdot (r,\vec \rho)\bigr)&=&(\vec\ell,l)q^{h_\al} \tp (r,\vec \rho)
&=&(\vec\ell,l)\tp (r,\vec \rho)q^{(\eta_{l}-\eta_{r})}.
\end{array}
\label{loc_op_def}
\ee
These formulas can be derived upon  observation that $\nbl_{l,r}$ acts non-trivially on at most one concatenation factor of length $1$ or $0$ in   $\vec m$.
The Cartan factors on the right result from  pushing $h_\al$ to the right, in accordance with (\ref{thetas1}) and (\ref{thetas2}).
\begin{remark}
\em
Let us comment on the item 2 above. This situation may occur when
\begin{itemize}
  \item $i=r$, then $\prt_{l,r}$ kills  the first  concatenation factor in $\vec m$ by (\ref{prt on phi}),
  \item there is a pair of adjacent nodes in $\vec m$, either $(l,k)$ or $(m,r)$, such that $r\not \succ k$ or, respectively, $m\not \succ l$. Then
  $\prt_{l,r}$ kills such factors, according to (\ref{prt on phi}). This is a manifestation of strict  triangularity of matrix $F$, cf. Proposition \ref{strong-triangular}.
\end{itemize}
\end{remark}

 Earlier we introduced two maps  $p_\Phi\colon \Phi_X\to B^-$ and $p_\Psi\colon \Psi_X\to \Vc$.
Define a new map $p_{\Phi\Psi}\colon \Phi_X\tp_{ \hat U^0}\Psi_X \to \Vc$ as  the composition of
$p_\Phi\tp p_\Psi$ with the action of $B^-$ on $\Vc$. It is  a $\hat U^0$-bimodule homomorphism.
Next we prove a key fact facilitating the method of Hasse diagrams.
\begin{propn}
\label{intertwing loc}
For each $\al\in \Pi$,
$$
p_{\Phi\Psi}\circ \sum_{(l,r)\in P(\al)}\pi^\al_{lr}\nbl_{l,r}=(e_\al\>\tr\>)\circ p_\Psi,
$$
where $\tr $ designates the action of $U^+$ on $\Vc$.
\end{propn}
\begin{proof}
Formulas (\ref{intertwiner_F}) and (\ref{can_invariant})
can be specialized as the following  system of equalities for  a simple pair $(l,r)\in P(\al)$:
\be
\begin{array}{ccccc}
[e_{\al},\phi_{lr}]&=&  \pi^\al_{lr}\>[h_\al]_q, & &    \\[5pt]
[e_{\al},\phi_{ir}]&=&  \sum_{(l,r)\in P(\al)} \pi^\al_{lr}\>\phi_{il}q^{h_{\al}} ,   \\[5pt]
[e_{\al},\phi_{lj}]&=&     -\sum_{(l,r)\in P(\al)} \pi^\al_{lr}\> q^{-h_{\al}} \phi_{rj},    \\[5pt]
[e_{\al},\phi_{ij}]&=&  0,  & i\not = l,j\not =r, &      \\[5pt]
e_{\al}\tr \psi_{j}&=&  -\dt_{jl}\sum_{(l,r)\in P(\al)} \pi^\al_{lr}\>q^{-h_{\al}}\psi_{r}. &
\end{array}
\label{basic_relation}
\ee
The summation in the second line is over $l$ subject to $i\succ l$, while in the third line over $r$ subject to $r\succ j$.
We can restrict our consideration to a basis element of $\Psi_X$ because all maps involved in the equality   are homomorphisms of right $\hat U^0$-modules.

Let $\vec m$ be a route.
If $|\vec m|=0,1$, then the assertion follows from comparing (\ref{basic_relation}) with  (\ref{prt on phi}) and (\ref{prt on psi}).
Suppose that the statement is proved for $\vec m$ of length $m\geqslant 1$ and let
 $\vec m=(i,k,\vec \rho)$ be a route with $\vec \rho\not =\varnothing$. Then $\psi_{\vec m}=\phi_{i,k}\psi_{(k,\vec \rho)}$, and
 the Leibnitz rule gives
$$
e_\al\tr \psi_{\vec m}=[e_\al, \phi_{i,k}]\psi_{(k,\vec \rho)}+\phi_{i,k}(e_\al\tr\psi_{(k,\vec \rho)}).
$$
Now the proof follows from Lemma \ref{diff1} and the induction assumption.
\end{proof}
\begin{remark}
\em
  The rationale for introducing an auxiliary module $\Psi_X$ is  that  $\nbl_{l,r}$ cannot be consistently defined  directly on $\Vc$ because
  $\psi_{\vec m}\in \Vc$ are generally not independent.

  Concatenation can be extended to $\Phi_X$ and $\Psi_X$ as an associative multiplication with zero product of non-adjacent routes.
  However the maps $p_\Phi$ and $p_\Psi$ will not be algebra and module homomorphisms.
\end{remark}
\section{Construction of the normalizer}
\label{SecNormNeg}
We assume that the Hasse diagram $\Hg(X)$ is a disjoint union of connected components bounded from below, e.g. if
 $X$ is a locally finite $U$-module.

For each pair  $i\succ j$ of nodes from $\Hg(X)$ define
$B_{j}^i\in \hat U^0$ by formulas
\be
 B_{j}^i=\frac{q^{-(\eta_{i}-\eta_{j})}}{[\eta_{i}-\eta_{j}]_q}, \quad (B_{j}^i)^{-1}=\frac{q^{2(\eta_{i}-\eta_{j})}-1}{q-q^{-1}}.
\label{Cartan_factors}
\ee

For each route $(i,\vec m)$ define
$$
  B^i_{\vec m}=B^i_{m_1}\ldots B^i_{m_k}\in \hat U^0, \quad \widetilde{(i,\vec m)} ={(i,\vec m)} B^i_{\vec m}\in \Psi_X,\quad \tilde \psi_{(i,\vec m)} =\psi_{(i,\vec m)} B^i_{\vec m}\in \Vc.
$$
If $\vec m=\varnothing $, then $\widetilde{(i)}=(i)$ and $\tilde \psi_{i}=\psi_i$  is understood.

For each simple pair $(l,r)$ we define special elements of $\Psi_X$ which we call  $(l,r)$-chains.
They will be sums  of
$\widetilde{\vec m}$
that are classified by a position of $\vec m$ relative to the nodes $\{l,r\}$.
All summands in a  chain carry the same weight and the routes involved have the same start and end nodes.

A 3-chain is defined as
\be
\widetilde{(\vec \ell, l,\vec \rho)}
+
\widetilde{(\vec \ell, l,r,\vec \rho)}
+
 \widetilde{(\vec \ell, r,\vec \rho)} ,
 \quad \vec \ell \not =\varnothing .
\label{3-chain_AA}
\ee
The route $\vec \rho$ may be empty while $\vec \ell$ may not.

A 2-chain  is defined as
\be
\widetilde{(l,\vec \rho)} +\widetilde {(l,r,\vec \rho)}.
 \label{2-chains}
\ee
As before, $\vec \rho$ may be empty, which corresponds to a chain $(l)+ (l,r)B^l_r$.

Finally, a 1-chain is   $\widetilde{\vec m}$ of the following two types:
\begin{itemize}
  \item $\vec m$ does not intersect with $(l,r)$,
  \item $\vec m\cap (l,r)\not =\varnothing$ but $\vec m \cup (l,r)$ is not a route $\max(\vec m)\dashleftarrow $.
\end{itemize}
The second option  is possible if $\vec m$ intersects  $(l,r)$ by exactly one node.

We  say that a route $\vec m$ belongs to a chain if $\widetilde{\vec m}$ enters it as a summand.

A chain in $\Vc$ is a $p_\Psi$-image of a chain in $\Psi_X$.
\begin{lemma}
\label{split_to_chains}
  Let $i\in I_X$ and $(l,r)$ be a simple pair. Then each route $i \dashleftarrow $ falls into exactly one $(l,r)$-chain.
\end{lemma}
\begin{proof}
Suppose that a route $i\stackrel{\>\>\vec m}{\dashleftarrow}$ is not a 1-chain.
  Then $\vec m\cap  (l,r)\not=\varnothing$ and $\vec m\cup (l,r)$ is a route  $i\stackrel{\>\>\vec m}{\dashleftarrow}$.
There are the following two possibilities.

If   $l=i$, then there is a $2$-chain  comprising $\vec m\cup r$ and  $(\vec m\cup r)\backslash \{r\}$.

If  $i\succ l$,  consider  two routes $\vec \ell\subset \vec m$ and $\vec \rho \subset \vec m$ consisting of nodes
on the left of  $l$ and on the right of $r$, respectively. While $i\in \vec \ell$, the route  $\vec \rho$ may be empty.
By assumption, $(\vec \ell, l,r,\vec\rho)$ is a route  $i{\dashleftarrow}$, then  so are $(\vec \ell,l,\vec\rho)$ and $(\vec \ell, r,\vec \rho)$.
They form a 3-chain  of type (\ref{3-chain_AA}) containing  $\vec m$.
 \end{proof}
\begin{lemma}
 The operator $\nbl_{l,r}$  kills all $(l,r)$-chains.
\label{chains_killed}
\end{lemma}
\begin{proof}
This is obvious with regard to  1-chains.
Applying $\nbl_{l,r}$ to a 2-chain (\ref{2-chains}) we get, up to a factor,
\be
1\tp (r,\vec \rho )(-q^{-(\eta_{l}-\eta_{r})}+  [\eta_{l}-\eta_{r}]_qB_{r}^l),
\nn
\ee
that is zero by definition of $ B_{r}^l$.

Application of $\nbl_{l,r}$ to 3-chain  (\ref{3-chain_AA})  with $i=\max(\vec \ell)$ produces
\be
(\vec \ell,l)\tp  (r,\vec \rho)\bigl(-q^{-(\eta_{l}-\eta_{r})}B_{l}^i
+
[\eta_{l}-\eta_{r}]_qB_{l}^iB_{r}^i
+
 q^{(\eta_{l}-\eta_{r})}B_{r}^i\bigr),
\label{3-chain factor}
\ee
up to a Cartan  factor on the right.
In order to see that the term in the brackets vanishes, it is convenient to divide it by $B_{l}^iB_{r}^i$  and use the formula   (\ref{Cartan_factors})
for their inverses.
This completes the proof.
\end{proof}
\begin{thm}
  For each $i\in I_X$, the element
\be
z_i=\psi_{i}+\sum_{ \vec m \not =\varnothing }\tilde \psi_{(i,\vec m)} \in \Vc
\label{norm_element}
\ee
belongs to $\Vc^+$. Furthermore, $ z_i=\wp \psi_i$.
\label{Main_thm}
\end{thm}
\begin{proof}
In order to prove that $z_i$ is $U^+$-invariant, it is sufficient to check  that each $\nbl_{l,r}$  annihilates  the pre-image of $z_i$ in $\Psi_X$, thanks to
Proposition \ref{intertwing loc}.
By Lemma \ref{split_to_chains}, the right-hand side  can be rearranged in a sum over $(l,r)$-chains,
which are all killed by $\nbl_{l,r}$, according to Lemma \ref{chains_killed}. This proves the first assertion.

The second assertion is obvious  because  $\tilde \psi_i=\psi_i \mod \g_- \Ac$ and $\wp (\g_- \Ac) =0$.
\end{proof}
\noindent
The summation in (\ref{norm_element}) is bounded from below due  to an assumption on $X$; the route $\vec m$ in summation is assumed non-empty.

\section{Canonical elements for $\Ac=U(\a)$}
\label{SecNormPos}
In this section we give a construction of canonical elements for a special case when   $\Ac$ is the universal enveloping algebra
of a Lie algebra $\a\supset \g$.

Choose root vectors $f_\al,e_\al$ in $\g$ normalized as  $(e_\al, f_\al)=1$ for all $\al\in \Rm^+$.
The tensor $\Fc$ turns to $\Cc=\sum_{\al\in \Rm^+}e_\al\tp f_\al$, in the classical limit $q\to 1$.

Since the adjoint action of $\g$ on $\a$ is semi-simple, we can choose
a $\g$-submodule $Y\subset \a$ complementary to $\g$. Denote by  $X$  its dual module.
 Pick up a weight basis $\{v_i\}_{i\in I_X}\subset X$ relative to $\h\subset \g$ and let $\{\psi_i\}_{i\in I_X}\subset Y $ be
its dual basis. We can choose it  compatible with an  irreducible decomposition of $X$.
The set $\{\psi_i\}_{i\in I_X}$   generates a PBW basis of $\Ac$ as a  $U^- - \B^+$-bimodule.

For each $\al \in \Rm^+$ let $\pi^\al_{ij}$ be matrix entries of the operator $\ad \>e_\al$ restricted to $X$.
Then the matrix $F\in \End(X)\tp U(\g_-)$ has entries  $\phi_{ij}=\pi^\al_{ij}f_\al$,  for all $i,j\in I_X$ such that $\nu_i-\nu_j=\al\in \Rm^+$.
\begin{propn}
\label{M-Z-generators_cl_gen}
  The elements
\be
z_i=\psi_{i}+\sum_{k=1}^\infty \sum_{i\succ i_1\succ \ldots \succ i_k} \phi_{ii_1}\ldots \phi_{i_{k-1}i_k}\psi_{i_k}\frac{1}{\eta_i-\eta_{i_1}}\ldots \frac{1}{\eta_i-\eta_{i_k}},
\quad i\in I_X,
\ee
generate a PBW basis in $\hat Z(\a,\g)$ over $\hat U(\h)$.
\end{propn}
\begin{proof}
The elements $\psi_i$, $i\in I_X$, generate a PBW basis of $\Ac$ as a $U^--B^+$-bimodule.
The map $\varphi \colon \psi_i\mapsto \wp \psi_i \wp$ from Section \ref{SecMickAlg} takes it to a PBW basis in $\wp\hat \Ac\wp\simeq \Ac/(\Jc_-+\Jc_+)$
as a right $U_q(\h)$-module. The algebra $\wp\hat \Ac\wp$ is isomorphic to $\hat Z(\a,\g)$ by Theorem \ref{thmMick}.
On the other hand,
$$
 \wp \hat\Ac\wp\simeq \wp (\hat \Ac/\Jc^+)=\wp \Vc.
$$
Thus the set $\{\wp \psi_i\}_{i\in I_X}$ is a PBW system in $\hat Z(\a,\g)$. But $\wp \psi_i=z_i$ for all $i\in I_X$  by Theorem \ref{Main_thm}.
This completes the proof.
\end{proof}
Remark that the summation parameter of the left sum does not exceed
 the length of the longest path in  the irreducible submodule
 containing $v_i$.
\subsection{Levi subalgebras}
In this section we consider in a more detail the case when $\a$ is simple and $\g$ is the commutant of a Levi subalgebra in $\a$, that is, $\g^\pm\subset \a^\pm$ and $\Pi_\g\subset \Pi_\a$ for their sets of simple roots.
The Cartan subalgebra $\k$ of $\a$ is the direct sum of $\h$ and $\c$ that commutes with $\g$.
We extend $U(\k)$ to $\hat U(\k)=\hat U(\h)U(\c)$.

Define by $\bar \Rm_\a$ the set of classes of roots whose root vectors enter the same irreducible $\g$-module in $\a^\pm\ominus \g^\pm$.
We will also write  $\bar \al \in \bar \Rm_\a$ for the class containing $\al\in \Rm_{\a/\g}$ (the complement $\Rm_\a\backslash \Rm_\g$).

The $\g$-modules  $X^+_{\bar \mu}$ and $X^-_{\bar \mu}$ defined by
$$
 X_{\bar \mu}^\pm=\sum_{\al\in \bar \mu} \a_{\pm\al}, \quad \bar \mu\in  \bar \Rm^+_{\a/\g}.
$$
are dual to each other.
We set
$$
\psi^-_\al=e_{-\al}\in X^-_{\bar \mu}\simeq (X^+_{\bar \mu})', \quad \psi^+_{-\al}= e_\al \in X^+_{\bar \mu}\simeq (X^-_{\bar \mu})', \quad \al\in \bar \mu\in \bar \Rm^+_{\a/\g}.
$$
Recall that by $X'$ we denote the dual to a $\g$-module $X$.
For each $\al\in \Rm^+_\g$ and $\mu\in \Rm_{\a/\g}^+$ let
  $N^\pm_{\al,\mu}$ be a structure constant determined from the equality
$$
   [e_{\al},e_{\pm\mu}]=
\left \{
\begin{array}{ccc}
N^\mp_{\al,\mu} e_{\pm\mu+\al}, & \mu \pm \al,\mu\in \Rm_{\a/\g}^+,\\0, &\mbox{otherwise}.
\end{array}
\right.
$$

For $\mu\in \Rm^+_{\a/\g}$ put
\be
z^\pm_\mu=e_{\pm\mu}+\sum_{k=1}^\infty\sum_{\nu\mp(\al_1+\ldots+\al_k)=\mu}f_{\al_1}\ldots f_{\al_k}e_{\pm\nu }\prod_{j=1}^k\frac{N^\pm_{\al_j,\mu_j}}{\eta_{\mp\mu}-\eta_{\mp\mu_j}},
\label{M-Z-generators}
\ee
where the right  summation is done over all $\nu\in \bar\mu$ and ordered $k$-partitions of $\pm(\nu-\mu)$ into
a sum  $\sum_{j=1}^{k}\al_j$ of positive roots from $\Rm^+_\g$ subject to
$$
\mu_j=\mu\pm(\al_1+\ldots+\al_j)\in \bar \mu, \quad j=1,\ldots, k.
$$

Here is a specialization of Proposition \ref{M-Z-generators_cl_gen}.
\begin{thm}
 The elements $z^\pm_\mu$,  $\mu \in \Rm_{\a/\g}^+$,  generate  a PBW basis in $\hat Z(\a,\g)$, over $\hat U(\k)$.
\label{PBW-base-class}
\end{thm}
\begin{proof}
The elements  $\psi^\pm_\mu$ generate a PBW basis in $U(\a)/(\g_-U(\a)+U(\a)\g_+)$. Now the proof is obvious because   $z^\pm_\mu=\wp_\g e_{\pm \mu}$ modulo $U(\a)\g_+$.
\end{proof}

\section{Canonical elements for quantum groups}
\label{SecGenerators}
Here we give a quantum version of the previous section, when  $\g$ is a commutant of a Levi subalgebra in $\a$.
As before, the invariant inner product on $\g$ is restricted from $\a$.
The algebra $\hat Z_q(\a,\g)$ enjoys a natural polarization: it is freely generated over $\hat U_q(\k)$ by a product of its  positive and negative parts.
We will need the respective versions of canonical elements participating in the PBW basis in $\hat Z_q(\a,\g)$.

We choose a quasitriangular structure in the quantum group $U_q(\a)$ in the same polarized form as in $U_q(\g)$ and use the same notation for it.
Define intertwining elements
$$
\check{\Ru}= q^{-\sum_{i}h_i\tp h_i}\Ru, \quad \tilde \Ru= q^{-\sum_{i}h_i\tp h_i}\Ru^{-1}.
$$
Suppose that $V$ is a $U_q(\a)$-module and $v_0\in V$ is  $U_q(\b_+)$-invariant.
Denote by $\pi$ the  representation homomorphism $\pi\colon U_q(\a)\to \End(V)$.
Introduce quantum Lax operators
$$
L^-=(\pi \tp \id)(\check{\Ru})\in U_q(\a_-), \quad L^+=(\id \tp \pi)(\tilde \Ru)\in U_q(\a_+)U_q(\k).
$$
In a weight basis $\{x_i\}_{i\in I_V}$ in $V$, their matrix entries carry  weights
$$
\wt(L^-_{ij})=\nu_j-\nu_i<0, \quad \wt(L^+_{ij})=\nu_j-\nu_i>0.
$$
Here $\nu_i\in \La(V)$ is a weight of the vector $x_i$.
Remark that we consider the same module $V$ for $L^+$ and $L^-$ for the sake of simplicity. One can take for them different modules independently. Set
$$
 \psi^-_{i}=\frac{1}{q-q^{-1}}L^-_{i0}\quad \mbox{for}\> \nu_i>\nu_0, \quad \psi^+_{i}=\frac{1}{q-q^{-1}}L^+_{i0} \quad \mbox{for}\> \nu_i<\nu_0.
$$
They are elements  of weight  $\nu_0-\nu_i$ (relative to $\a$) in the positive and negative  Borel subalgebras of $U_q(\a)$, respectively.
\begin{propn}
\label{Mick_gen}
Let $V$ be a finite dimensional $U_q(\a)$-module with a $U_q(\b_+)$-invariant vector $v_0$. Suppose that $X^\ve\subset V$ with  $\ve=\pm$ is an irreducible
$U_q(\g_+)$-submodule whose weights satisfy $\ve\La(X^\ve)<0$.
  Then  $\psi_{X^\ve}=\sum_{i\in I_{X^\ve}} x_i\tp  \psi^\ve_i $ is a canonical invariant element  of $X^\ve$.
\end{propn}
\begin{proof}
 The intertwining axiom of the R-matrix implies identities
\be
 (e_{\al}\tp 1 + q^{-h_{\al}}\tp e_{\al}) \tilde \Ru&=& \tilde \Ru (e_{\al}\tp 1 + q^{h_{\al}}\tp e_{\al}) ,
\nn\\
(e_{\al}\tp q^{-h_{\al}} + 1\tp e_{\al})\check{\Ru}
&=&  \check{\Ru}(e_{\al}\tp q^{h_{\al}} + 1\tp e_{\al}),
\nn
\ee
for all simple roots from $\Pi_\a$, and for all $\al\in \Pi_\g$ in particular.
These readily imply
\be
 e_{\al}\psi^+_i&=&- q^{-h_{\al}}\sum_{k\prec i\prec 0} \pi(e_{\al})_{ik}  \psi^+_{k} \mod U_q(\a)\g_+,
\nn
\\
 e_{\al}\psi^-_{i}&=&- q^{-h_{\al}}\sum_{0\prec k\prec i}\pi(e_{\al})_{ik} \psi^-_{k} \mod U_q(\a)\g_+,
\nn
\ee
for all $\al \in \Pi_\g$,
because all matrix entries $\pi(e_{\al})_{k0}$ vanish. Restricting $i$ to $I_{X^\ve}$ we obtain components of the canonical element relative to $X^\ve$.
\end{proof}
\noindent
Note that  weights are $\wt(\psi^\pm_i)=-\nu_i$, with respect to   $\h\subset \g$ because $\nu_0$ is orthogonal to $\h^*$.

Under the assumptions on $\a$ and $\g\subset \a$ of the previous section, set $\Ac=U_q(\a)$ and $U=U_q(\g)$.
Set $V$ to be the irreducible $\Ac$-module whose highest weight is the maximal root in $\Rm^+_\a$.
It is a deformation of the adjoint module $\a$, and its $\g$-module structure is similar to the
case $q=1$. The problem essentially reduces to explicit description of the relative Lax operators.

The weights   in $\La(V)$ have multiplicities  $\dim V[0]=\rk \>\a$ and $\dim V[\pm \al]=1$ for $\al\in \Rm_\a^+$. Pick up a non-zero vector  $e^\pm_\al$ in each $V[\pm\al]$.
One can find a basis $\{h^\al\}_{\al\in \Pi_\a}\in V[0]$ such that $e_{\pm\al} h^\bt=0$ for distinct $\al, \bt \in \Pi_\a$.
For each $\bar \mu\in \bar \Rm^+_{\a/\g}$ select $\bt \in \Pi_\a$ such that $e_{\pm \mu} h^\bt \propto \dt_{\mu\bt} e^\pm_\mu$, for all $\mu \in \bar \mu$.
Such $\bt$ depends only on the quasi-root $\bar \mu$.

The vector space $X^\pm _{\bar \mu}=\Span\{e^\pm_{\nu}\}_{\nu\in \bar \mu}\subset V$ is a $U_q(\g)$-module.
For the fixed basis $\{e^\mp_\mu\}_{\mu\in \bar \mu}\subset  X^\mp_{\bar \mu} $, construct the invariant canonical element with components $\{\psi^\pm_\mu\}_{\mu\in \bar \mu}$ as prescribed
by Proposition \ref{Mick_gen}, and the set  $\{z^\pm_\mu\}_{\mu\in \bar \mu}\subset  \hat Z_q(\a,\g)$, as in Theorem \ref{Main_thm}.

Now consider the localization of the quantum group $U_q(\a)$ to a $\C[[\hbar]]$-Hopf algebra $U_\hbar(\a)$, assuming $q=e^\hbar$ and denote
by $\hat Z_\hbar(\a,\g)$ the corresponding extension of  $\hat Z_q(\a,\g)$.
\begin{propn}
The elements $z^\pm_\mu$, $\mu\in  \Rm_{\a/\g}^+$ generate a PBW basis in $\hat Z_\hbar(\a,\g)$ over $\hat U_\hbar(\k)$.
\end{propn}
\begin{proof}
  Observe that weight spaces in $\hat Z^\pm_q(\a,\g)$ are finite dimensional, and their dimension is independent of $q$.
  The elements $z^\pm_\mu$ deliver a PBW system in the classical limit,  by Theorem \ref{PBW-base-class}, which proves the assertion.
\end{proof}
The problem of PBW basis at a particular $q$ away from a root of unity boils down to the question whether $\psi^\pm_\mu$, $\mu\in \Rm^+_{\a/\g}$,
generate a PBW basis in $U_q(\a)/(\g_-U_q(\a)+U_q(\a)\g_+)$ over the Cartan subalgebra. This is a question about properties of the quantum L-operator and its matrix elements.
We conjecture that is the case indeed.

\vspace{20pt}
\noindent

\underline{\large \bf Acknowledgement}
\vspace{10pt} \\
This work was done at the Center of Pure Mathematics MIPT.
It is financially supported  by Russian Science Foundation  grant 23-21-00282.

\subsection*{\underline{Declarations}}

\subsection*{Ethical approval}
Not applicable

\subsection*{Author's contribution}
Both authors have made equal contribution to the current research and to preparation of the manuscript.

\subsubsection*{Data Availability}
 Data sharing not applicable to this article as no datasets were generated or analysed during the current study.
\subsubsection*{Funding}
This research is  supported  by Russian Science Foundation  grant 23-21-00282.
\subsubsection*{Competing interests}
The authors have no competing interests to declare that are relevant to the content of this article.

\end{document}